\documentclass{amsart}
\usepackage{amssymb,amstext,amsmath,amscd,amsthm,amsfonts,enumerate,graphicx,latexsym,stmaryrd,multicol}
\usepackage{setspace}
\usepackage{bm}
\usepackage[all]{xy} 
\newtheorem{thm}{Theorem}[section]
\newtheorem{lem}[thm]{Lemma}
\newtheorem{prop}[thm]{Proposition}
\newtheorem{cor}[thm]{Corollary}
\theoremstyle{definition}
\newtheorem{dfn}[thm]{Definition}

\newtheorem{rem}[thm]{Remark}

\newtheorem{ex}[thm]{Example}

\theoremstyle{remark}

\newtheorem*{ac}{Acknowledgments}
\newtheorem*{conv}{Convention}
\newtheorem*{proof of claim}{Proof of Claim}

\numberwithin{equation}{thm}

\def\ann{\operatorname{Ann}}
\def\cm{\mathsf{CM}}
\def\cok{\operatorname{Coker}}
\def\depth{\operatorname{depth}}
\def\Ext{\operatorname{Ext}}

\def\ge{\geqslant}
\def\height{\operatorname{ht}}
\def\Hom{\operatorname{Hom}}
\def\image{\operatorname{Im}}

\def\ker{\operatorname{Ker}}
\def\le{\leqslant}

\def\m{\mathfrak{m}}

\def\P{\mathbb{P}}
\def\p{\mathfrak{p}}
\def\pd{\operatorname{pd}}
\def\Q{\mathbb{Q}}
\def\q{\mathfrak{q}}

\def\spec{\operatorname{Spec}}
\def\supp{\operatorname{Supp}}

\def\Tor{\operatorname{Tor}}

\def\V{\mathrm{V}}

\def\gor{\mathsf{Gor}}
\def\fid{\mathsf{FID}}

\def\PD{\mathsf{PD}}
\def\reg{\mathsf{Reg}}
\def\id{\operatorname{id}}

\def\D{\mathrm{D}}
\def\Z{\mathbb{Z}}
\begin{document}
\allowdisplaybreaks
\title{Asymptotic behavior of homological invariants of localizations of modules}
\author{Kaito Kimura}
\address{Graduate School of Mathematics, Nagoya University, Furocho, Chikusaku, Nagoya 464-8602, Japan}
\email{m21018b@math.nagoya-u.ac.jp}
\thanks{The author was partly supported by Grant-in-Aid for JSPS Fellows Grant Number 23KJ1117.}
\begin{abstract}
Let $R$ be a commutative noetherian ring, $I$ an ideal of $R$, and $M$ a finitely generated $R$-module.
We consider the asymptotic injective dimensions, projective dimensions, Bass numbers, and Betti numbers of localizations of $M/I^n M$ at prime ideals of $R$ and prove that these invariants are stable or have polynomial growth for large integers $n$ that do not depend on the prime ideals.
\end{abstract}
\maketitle
\section{Introduction}

Throughout the present paper, all rings are assumed to be commutative and noetherian.
Let $R$ be a ring, $I$ an ideal of $R$, and $M$ a finitely generated $R$-module.
The asymptotic behavior of the quotient modules $M/I^n M$ of $M$ for large integers $n$ has been actively studied in commutative algebra.
Brodmann \cite{Br} proved that for any ideal $J$ of $R$, the grade of $J$ on $M/I^n M$ is stable for large $n$ depending on $J$.
In particular, when $R$ is local, the depth of $M/I^n M$ attains a stable constant value for all large $n$.
Using the openness of loci of modules, it is shown in \cite{Ki} under several assumptions that there exists an integer $k$ such that for all integers $n\ge k$ and all prime ideals $\p$, $\depth (M/I^n M)_\p=\depth (M/I^k M)_\p$.
In this situation, the grade of $J$ on $M/I^n M$ is stable for all $n\ge k$ and all ideals $J$ of $R$.
The difference from Brodmann's result above is that the integer $k$ is independent of $J$.
Such an integer $k$ exists if $R$ is an excellent ring, or a homomorphic image of a Cohen--Macaulay ring, or a semi-local ring.

The depth of a module is defined as the infimum of non-vanishing Bass numbers.
The injective dimension is the other important invariant defined as the supremum of non-vanishing Bass numbers.
As with the depth, the asymptotic stability of the injective dimension of $M/I^n M$ is known when $R$ is local.
Moreover, for all finitely generated $R$-modules $N$ and all integers $i$, the lengths of the modules $\Ext_{R}^i(N,M/I^n M)$ have polynomial growth for all large $n$ whenever they are finite; see \cite{Ko, T}.
The same holds dually for projective dimensions and lengths of Tor modules.

In this paper, we study the asymptotic injective dimensions, projective dimensions, Bass numbers, and Betti numbers of localizations of $M/I^n M$ at prime ideals of $R$.
For large integers $n$ that do not depend on the prime ideals, we consider whether these invariants are stable or have polynomial growth.
The first main result of the present paper is the following theorem; for the definition of an acceptable ring in the sense of Sharp \cite{S} see Definition \ref{def rings}.
Similar results exist for projective dimension; see Corollary \ref{main thm of pd}.

\begin{thm}[Corollaries \ref{main cor of fid} and \ref{main cor of fid local}]\label{main1 id stable}
Put $\bar{R}=R/(I+\operatorname{Ann}_R (M))$.
Then there is $k>0$ such that
$$
\id_{R_\p} (M/I^n M)_\p=\id_{R_\p} (M/I^k M)_\p
$$ 
for all integers $n\ge k$ and all prime ideals $\p$ of $R$ in each of the following cases.
\begin{enumerate}[\rm(1)]
\item $M$ or $M/I^n M$ has finite injective dimension for some $n>0$.
\item $\bar{R}$ is acceptable. (e.g. $\bar{R}$ is excellent or a homomorphic image of a Gorenstein ring.)
\item $\bar{R}$ is semi-local.
\end{enumerate}
\end{thm}

Obviously, we may replace all $\bar{R}$ in the above result with $R$.
In the proofs of the former two cases, we use the method developed in \cite{Ki} (see also \cite{RS}).
Using the openness of the finite injective dimension locus of a graded module, we give the asymptotic stability of the loci of the homogeneous components of the associated graded module $\bigoplus_{i\ge 0}I^i M/I^{i+1} M$, and $M/I^n M$.
To implement this argument, the theorems in \cite{K} about the openness of finite injective dimension loci need to be improved.
Also, some proofs require a different argument from the ones in \cite{Ki}.
In case (3), the claim is shown by attributing to the case where $R$ is a complete local ring, but it does not follow immediately from the case (2).
So first, we prove (1) of the theorem below for the case (i).
From the results, we can provide (ii) of Theorem \ref{main2 finite poly bass and betti}(1), and so (3) of Theorem \ref{main1 id stable}.

\begin{thm}[Corollary \ref{finite poly bass acceptable}, Theorems \ref{finite poly bass local} and \ref{finite poly betti}]\label{main2 finite poly bass and betti}
Let $s\ge 0$ and $\bar{R}=R/(I+\operatorname{Ann}_R (M))$.
\begin{enumerate}[\rm(1)]
\item Suppose either that {\rm(i)} $\bar{R}$ is acceptable and of finite dimension or that {\rm(ii)} $\bar{R}$ is a semi-local ring.
Then there exist polynomials $\varphi_1, \ldots, \varphi_l \in\Q[x]$ and $k>0$ such that the following condition is satisfied:
for any prime ideal $\p$ of $R$, there exists $1\le j\le l$ such that $\mu^{s} (\p,M/I^n M)=\varphi_j (n)$ for all $n\ge k$.
\item There exist polynomials $\varphi_1, \ldots, \varphi_l \in\Q[x]$ and $k>0$ such that the following condition is satisfied:
for any prime ideal $\p$ of $R$, there exists $1\le j\le l$ such that $\beta_{s} (\p,M/I^n M)=\varphi_j (n)$ for all $n\ge k$.
\end{enumerate}
\end{thm}

As mentioned above, the Bass and Betti numbers have polynomial growth for each prime ideal, but it is worth noting that their growths are represented by a finite number of polynomials.
The openness of loci also plays an important role in the proofs of both (1) and (2) of the above theorem.
The locus that appears in the proof of (2) is always open, while that in (1) is not always so in general.
The difference between the assumptions in (1) and (2) arises from this.

The organization of this paper is as follows. 
In Section 2, we study the openness of the finite injective dimension locus of a graded module and show Theorem \ref{main1 id stable} in the former two cases.
In Section 3, we prove Theorem \ref{main2 finite poly bass and betti} and the remaining case of Theorem \ref{main1 id stable} and consider some examples.

We close the section by stating our convention.

\begin{conv}
Let $R$ be a ring, and $M$ an $R$-module.
We denote by $\operatorname{Ann}_R (M)$ the annihilator ideal of $M$.
The injective dimension and the projective dimension of $M$ are denoted by $\id_{R}M$ and $\pd_{R}M$, respectively. 
For a property $\P$ of local rings, $\P(R)=\{\p\in\spec (R)\mid \P$ holds for $R_\p\}$ is called the $\P$-locus of $R$.
Similarly, for a property $\P$ of modules over a local ring, $\P_R (M)=\{\p\in\spec (R)\mid \P$ holds for $M_\p\}$ is called the $\P$-locus of $M$.
Let $\p$ be a prime ideal of $R$, $\kappa(\p)$ the residue field of $R_\p$, and $i\ge 0$ an integer.
Suppose that $M$ is finitely generated over $R$.
The number $\mu^{i}_R (\p,M)=\dim_{\kappa(\p)}\Ext_{R_\p}^i(\kappa(\p),M_\p)$ is called $i$-th Bass number of $M$ with respect to $\p$.
The number $\beta_{i}^R (\p,M)=\dim_{\kappa(\p)}\Tor^{R_\p}_i(\kappa(\p),M_\p)$ is called $i$-th Betti number of $M$ with respect to $\p$.
We omit subscripts/superscripts if there is no ambiguity.
\end{conv}

\section{Asymptotic injective dimensions}

In this section, we consider the necessary and sufficient condition for the finite injective dimension locus of a module to be an open set.
In that situation, we give the asymptotic stability of injective dimensions of localizations of modules.
The following notation is used in this paper.

\begin{dfn}
Let $R$ be a ring, $M$ an $R$-module, $I$ an ideal of $R$, and $f$ an element of $R$.
We set
\begin{itemize}
\item $\D_R(f) = \{\p\in\spec (R)\mid f\notin\p\}$.
\item $\D_R(I) = \{\p\in\spec (R)\mid I\nsubseteq\p\}$.
\item $\V_R(I) = \{\p\in\spec (R)\mid I\subseteq\p\}$.
\item $\fid_R (M) = \{\p\in\spec (R)\mid \id_{R_\p}M_\p<\infty \}$.
\item $\cm(R) = \{\p\in\spec (R)\mid$ The local ring $R_\p$ is Cohen--Macaulay$ \}$.
\item $\gor(R) = \{\p\in\spec (R)\mid$ The local ring $R_\p$ is Gorenstein$ \}$.
\end{itemize}
\end{dfn}

Let $R$ be a ring.
We say that a subset $U$ of $\spec (R)$ is \textit{stable under generalization} if every prime ideal $\q$ of $R$ such that $\q\subseteq \p$ for some $\p\in U$ belongs to $U$.
Note that $\cm(R)$, $\gor(R)$, and $\fid_R (M)$ are stable under generalization for any $R$-module $M$.
We state one lemma about subsets that are stable under generalization.
For a family of subsets that are stable under generalization, this lemma provides a method to determine whether they are equal.

\begin{lem}\label{Unified lemma}
Let $R$ be a ring, $\Lambda$ a set, and $n\in \mathbb{R}\cup\{\infty\}$.
Let $f:\spec(R)\times\Lambda\to \mathbb{R}\cup\{\pm \infty\}$ and $g:\spec(R)\to \mathbb{R}\cup\{\pm \infty\}$ be maps.
Put $X=\{\p\in\spec(R) \mid g(\p)<n \}$ and $X_{\lambda}=\{\p\in\spec(R) \mid f(\p, \lambda)<n \}$ for all $\lambda\in\Lambda$.
If the following three conditions hold, then $X=X_{\lambda}$ for all $\lambda\in\Lambda$.
\begin{enumerate}[\rm(1)]
\item For any prime ideal $\p$ of $R$, the equality $g(\p)={\rm sup}\{f(\p, \lambda) \mid \lambda\in\Lambda \}$ holds.
\item For any minimal element $\q$ of $\spec(R)\setminus X$, the function $f(\q,-):\Lambda\to\mathbb{R}\cup\{\pm \infty\}$ is constant.
\item  For any $\lambda\in\Lambda$, $X_{\lambda}$ is stable under generalization.
\end{enumerate}
\end{lem}

\begin{proof}
It is seen by (1) that $X\subseteq X_{\lambda}$ for all $\lambda\in\Lambda$.
Let $\p\notin X$.
Since $R$ is noetherian, there is a minimal element $\q$ of $\spec(R)\setminus X$ contained in $\p$.
It follows from (1) and (2) that $f(\q,\lambda)=g(\q)\ge n$ for any $\lambda\in\Lambda$.
This means that for all $\lambda\in\Lambda$, $\q$ does not belong to $X_{\lambda}$, and thus neither does $\p$ by (3).
\end{proof}

If $X$ is open, then $\spec(R)\setminus X$ has at most a finite number of minimal elements.
Although $X_{\lambda}$ are generally infinite sets, Lemma \ref{Unified lemma} says that we can compare them using only information about a finite number of prime ideals in that situation.
In the following, we consider the openness of finite injective dimension loci.
To state Theorem \ref{fid locus extension}, we now prepare several lemmas about injective dimension.
First, we list some well-known results.

\begin{lem}\label{id sum}
Let $R$ be a ring, $L$ and $M$ finitely generated $R$-modules, and $\p$ a prime ideal of $R$.
\begin{enumerate}[\rm(1)]
\item \cite[Proposition 3.1.14]{BH} There is the equality $\id_{R_\p}M_\p={\rm sup}\{i \mid \Ext_{R_\p}^i(R_\p/\p R_\p,M_\p)\ne0 \}$.
\item \cite[Theorem 3.1.17]{BH} If $M_\p$ is a nonzero module of finite injective dimension, then one has the equality $\id_{R_\p}M_\p=\depth R_\p$.
In particular, if $\supp_{R}(L)=\supp_{R}(M)$, then $L_\q$ and $M_\q$ have the same injective dimension for all prime ideals $\q$ of $R$ if and only if the equality $\fid_R (L)=\fid_R (M)$ holds.
\item Let $\{N_\lambda\}_{\lambda\in\Lambda}$ be a family of $R$-modules and $N=\bigoplus_{\lambda\in\Lambda} N_\lambda$.
Then $\id_R (N)={\rm sup}\{\id_R(N_\lambda) \mid \lambda\in\Lambda \}$.
\end{enumerate}
\end{lem}

Lemma \ref{id sum}(2) asserts that to prove Theorem \ref{main1 id stable}, it suffices to show the equalities of the finite injective dimension loci of $M/I^n M$ for all large $n$.
Our idea is to give these equalities using the openness of loci and Lemma \ref{Unified lemma}.
In the next two lemmas, for an $R$-algebra $S$ and an $R$-module M, we observe the Gorensteinness of $S$ and the finiteness of injective dimension of $M$.
Imposing the assumption on $\Ext_R^i(S,M)$, Lemma \ref{id lem1} asserts that the Gorensteinness of $S$ provides information about the finiteness of injective dimension of $M$, and Lemma \ref{id lem2} insists that the converse implication holds.
These lemmas play an essential role in the proof of Theorem \ref{fid locus extension}.

\begin{lem}\label{id lem1}
Let $R$ be a ring, $S$ an $R$-algebra, $M$ an $R$-module, and $n\ge 0$ an integer.
Put $t=\id_S (S)$.
Suppose that $\Ext_R^i(S,M)$ is a free $S$-module for all $0\le i\le n$ and $\Ext_R^j(S,M)=0$ for all $j>n$.
Then $\Ext_R^k(N,M)=0$ for any $k>n+t$ and any $S$-module $N$.
\end{lem}

\begin{proof}
Let $I:0\to I^0\to I^1\to\cdots$ be an injective resolution of $M$.
There is a complex
$$
0\to \Hom_R(S, I^0) \xrightarrow{d^0} \Hom_R(S, I^1) \xrightarrow{d^1} \Hom_R(S, I^2) \xrightarrow{d^2} \cdots
$$
of injective $S$-modules.
Also, for any $0\le l\le n$, there are exact sequences
$$
0 \to \ker{d^l} \to \Hom_R(S, I^l) \to \image{d^l} \to 0 \ \ {\rm and}\ \ 
0 \to \image{d^l} \to \ker{d^{l+1}} \to \Ext_R^{l+1}(S, M) \to 0
$$
of $S$-modules.
By assumption, for any $0\le i\le n$, the free $S$-module $\Ext_R^i(S,M)$ has injective dimension $t$.
We see by induction on $l$ that for any $0\le l\le n$, $\ker{d^l}$ and $\image{d^l}$ have injective dimension at most $t$.
Since $\Ext_R^j(S,M)=0$ for any $j>n$, the complex
$$
0\to \Hom_R(S, I^n) \xrightarrow{d^n} \Hom_R(S, I^{n+1}) \xrightarrow{d^{n+1}} \Hom_R(S, I^{n+2}) \to \cdots
$$
is an injective resolution of $\ker{d^n}$ as an $S$-module.
On the other hand, there is an isomorphism 
$$
\Hom_R(N,I)\simeq\Hom_S(N,\Hom_R(S,I))
$$ 
of complexes for any $S$-module $N$.
We get
$
\Ext_R^k(N, M)\simeq\Ext_S^{k-n}(N,\ker{d^n})=0
$
for every $k>n+t$ since $\ker{d^n}$ has injective dimension at most $t$.
\end{proof}

\begin{lem}\label{id lem2}
Let $R$ be a ring, and $S$ an $R$-algebra.
Let $M$ be an $R$-module, and $n\ge 0$ an integer.
Suppose that $\Ext_R^n(S,M)$ is a nonzero free $S$-module, and $\Ext_R^i(S,M)=0$ for all $i>n$.
If $n\le \id_R(M)$, then $\id_S(S)\le \id_R(M)-n$.
\end{lem}

\begin{proof}
We may assume $t:=\id_R(M)<\infty$.
Let $I:0\to I^0\to I^1\to\cdots \to I^t \to 0$ be an injective resolution of $M$.
Since $\Ext_R^i(S,M)=0$ for any $i>n$, the complex
$$
0\to \Hom_R(S, I^n) \xrightarrow{d} \Hom_R(S, I^{n+1}) \to \Hom_R(S, I^{n+2}) \to \cdots \to \Hom_R(S, I^t) \to 0
$$
is an injective resolution of $\ker{d^n}$ as an $S$-module.
In particular, $\ker{d^n}$ has injective dimension at most $t-n$.
There is the natural epimorphism from $\ker{d^n}$ to the nonzero free $S$-module $\Ext_R^n(S,M)$.
Hence $S$ is a direct summand of $\ker{d^n}$.
We have $\id_S(S)\le \id_S(\ker{d^n})\le t-n$.
\end{proof}

Recall the following remark on graded modules and Ext and Tor modules.
For instance, a similar argument to the remark exists in the proofs of \cite[Lemma 3.7]{Ki} and \cite[Proposition 4]{Ko}.

\begin{rem}\label{graded ext and f.g.}
Let $A=\bigoplus_{i\ge 0}A_i$ be a graded ring and let $M=\bigoplus_{i\in \Z}M_i$ be a graded $A$-module.
Let $N$ be a finitely generated $A_0$-module.
Then, for any integer $i\ge 0$, $\Ext_{A_0}^i(N,M)\simeq\bigoplus_{t\in\Z}\Ext_{A_0}^i(N,M_t)$ and $\Tor_i^{A_0}(N,M)\simeq\bigoplus_{t\in\Z}\Tor_i^{A_0}(N,M_t)$
are graded $A$-modules.
If $M$ is finitely generated, then so are these modules.
(Compute them using a resolution of N by free $A_0$-modules of finite rank.)
\end{rem}

The following theorem improves \cite[Theorems 3.6 and 4.4]{K}.
Indeed, if $A=A_0$, then $M$ is a finitely generated $A_0$-module.
Theorem \ref{fid locus extension} removes the assumptions imposed in \cite[Theorem 4.4]{K} that $M_\p$ is maximal Cohen--Macaulay.
While \cite[Proposition 2.4]{Ta}, used in the proof of \cite[Theorem 3.6]{K}, is proved using spectral sequences, Theorem \ref{fid locus extension} is done without using them.

\begin{thm}\label{fid locus extension}
Let $A=\bigoplus_{i\ge 0}A_i$ be a graded ring and let $M=\bigoplus_{i\in \Z}M_i$ be a finitely generated graded $A$-module.
Suppose that $\p\in\supp_{A_0}(M)\cap\fid_{A_0}(M)$.
Then the following conditions are equivalent.
\begin{enumerate}[\rm(1)]
\item $\gor(A_0/\p)$ contains a nonempty open subset of $\spec (A_0/\p)$.
\item $\fid_{A_0}(M)$ contains a nonempty open subset of $\V_{A_0}(\p)$.
\end{enumerate}
\end{thm}

\begin{proof}
We can freely replace our ring $A$ with its localization $A_f$ for any element $f\in A_0\setminus\p$ to prove the theorem; see \cite[Lemma 2.5]{K}.
Since $\p\in\supp_{A_0}(M)\cap\fid_{A_0}(M)$, there exists $j\in \Z$ such that $(M_j)_\p$ is a nonzero finitely generated $(A_0)_\p$-module of finite injective dimension.
It follows from \cite[Corollary 9.6.2 and Remark 9.6.4(a)]{BH} that $(A_0)_\p$ is Cohen--Macaulay.
We have $n:=\id_{(A_0)_\p}(M_\p)=\height\p$; see (2) and (3) of Lemma \ref{id sum}.
It follows from Remark \ref{graded ext and f.g.}, \cite[Lemma 8.1]{HR} and \cite[Lemma 1.1.3(1)]{RS} that we may assume that for any $0\le i\le n$, $\Ext_{A_0}^i(A_0/\p, M)$ is free as an $A_0/\p$-module and $\Ext_{A_0}^{n+1}(A_0/\p, M)$ is the zero module.
By \cite[(3) and (4) of Lemma 2.7]{K}, we may assume that there exists an $A_0$-regular sequence $\bm{x}=x_1,\ldots,x_n$ in $\p$ and that $\p^r$ is contained in $\bm{x} A_0$ for some integer $r>0$.
Set $L_i=\p^i(A_0/\bm{x} A_0)$ for each $0\le i\le r$.
Thanks to \cite[Lemma 2.7(6)]{K}, we may assume that $L_{i-1}/L_{i}$ is a free $A_0/\p$-module for each $1\le i\le r$.
This means that for each $j>0$, if $\Ext_{A_0}^{j} (A_0/\p, M)$ is the zero module, then so is $\Ext_{A_0}^{j} (\p/\bm{x} A_0, M)$.
Since $\bm{x}$ is an $A_0$-regular sequence, $\Ext_{A_0}^i(A_0/\bm{x} A_0, M)=0$ for any $i>n$.
By induction on $i$, we have $\Ext_{A_0}^{i+1} (A_0/\p, M)\simeq\Ext_{A_0}^{i} (\p/\bm{x} A_0, M)=0$ for any $i>n$.

(1) $\Rightarrow$ (2): 
We may assume that $A_0/\p$ is Gorenstein.
Let $\q\in\V_{A_0}(\p)$ and $t=\id_{(A_0/\p)_\q}(A_0/\p)_\q$.
Lemma \ref{id lem1} yields that $\Ext_{(A_0)_\q}^k((A_0/\q)_\q,M_\q)=0$ for any $k>n+t$, which means that $\id_{(A_0)_\q}(M_i)_\q \le n+\height(\q/\p)$ for all $i\in \Z$ by Lemma \ref{id sum}(1).
Lemma \ref{id sum}(3) deduces that $\q$ belongs to $\fid_{A_0}(M)$.

(2) $\Rightarrow$ (1): 
We may assume that $\fid_{A_0}(M)$ contains $\V_{A_0}(\p)$.
By (1) and (3) of Lemma \ref{id sum}, the free $A_0/\p$-module $\Ext_{A_0}^n (A_0/\p, M)$ is nonzero.
Let $\q\in\V_{A_0}(\p)$.
Note that $n=\height\p\le\height\q=\id_{(A_0)_\q}(M_\q)<\infty$.
It follows from Lemma \ref{id lem2} that we have inequalities $\id_{(A_0/\p)_\q}(A_0/\p)_\q \le \id_{(A_0)_\q}(M_j)_\q-n< \infty$.
It means that $\q/\p$ is in $\gor(A_0/\p)$.
\end{proof}

Below is a direct corollary of Theorem \ref{fid locus extension}.

\begin{cor}\label{cor fid locus2}
Let $A=\bigoplus_{i\ge 0}A_i$ be a graded ring and let $M=\bigoplus_{i\in \Z}M_i$ be a finitely generated graded $A$-module.
Suppose that $\gor(A_0/\p)$ contains a nonempty open subset of $\spec (A_0/\p)$ for any prime ideal $\p$ of $A_0$ belonging to $\supp_{A_0}(M)\cap\fid_{A_0}(M)$.
Then $\fid_{A_0}(M)$ is an open subset of $\spec(A_0)$.
\end{cor}

\begin{proof}
It follows from \cite[Lemma 1.1.3(1)]{RS} that $\supp_{A_0}(M)$ is closed.
Hence the corollary is shown analogously as in the proof of \cite[Corollary 3.7(1)]{K}; replace \cite[Theorem 3.6]{K} with Theorem \ref{fid locus extension} in the proof of \cite[Corollary 3.7(1)]{K}.
\end{proof}

Applying Lemma \ref{Unified lemma} to a family of finite injective dimension loci shows the asymptotic stability of these loci of the homogeneous components of a graded module.

\begin{prop}\label{homog comp fid locus}
Let $A=\bigoplus_{i\ge 0}A_i$ be a homogeneous graded ring and $M=\bigoplus_{i\in \Z}M_i$  a finitely generated graded $A$-module.
Denote by $N_t$ the graded $A$-submodule $\bigoplus_{i\ge t}M_i$ of $M$ for each $t\in\Z$.
If $\fid_{A_0}(N_t)$ is open for all $t\in \Z$, then there is an integer $k\in \Z$ such that $\fid_{A_0}(M_n)=\fid_{A_0}(M_k)$ for all $n\ge k$.
\end{prop}

\begin{proof}
The open subset $\fid_{A_0}(N_t)$ of $\spec(A_0)$ is contained in $\fid_{A_0}(N_{t+1})$ for any $t\in \Z$.
There is an integer $m\in \Z$ such that $X:=\fid_{A_0}(N_m)=\fid_{A_0}(N_t)$ for all $t\ge m$ since $A_0$ is noetherian.
Let $Y$ be the set of minimal elements of $\spec(R)\setminus X$.
As $X$ is open, $Y$ is a finite set.
We take an integer $s$ such that $s>\height\q$ for all $\q\in Y$.
It follows from Lemma \ref{id sum}(2), \cite[Theorem 1.1]{FFGR} and \cite[Theorem 2]{Rob} that for any $\q\in Y$, $\id_{A_0} (M_n)_\q=\infty$ if and only if $\q$ contains $\ann_{A_0}(\Ext_{A_0}^s(A_0/\q,M_n))$.
A similar argument to the proof of \cite[Lemma 3.7]{Ki} shows that there exists an integer $k\ge m$ such that for all $n\ge k$ and all $\q\in Y$, $\ann_{A_0}(\Ext_{A_0}^s(A_0/\q,M_n))=\ann_{A_0}(\Ext_{A_0}^s(A_0/\q,M_k))$.
This means that for any $n\ge k$ and any $\q\in Y$, $\id_{A_0} (M_n)_\q=\id_{A_0} (M_k)_\q$; see Lemma \ref{id sum}(2).
Applying Lemma \ref{Unified lemma} to $\Lambda=\{n\in\Z \mid n\ge k\}$, $n=\infty$, $f(\p, \lambda)=\id_{A_0} (M_\lambda)_\q$ and $g(\p)=\id_{A_0} (N_k)_\q$, we get $X=\fid_{A_0}(M_n)$ for all $n\ge k$.
\end{proof}

Now we are ready to give a proof of the main result of this section.
The structure of this section is constructed with reference to that of \cite{Ki}.
However, many properties of depth were used in the proof of the main result of \cite{Ki}, such as the fact that the depth of a nonzero module is always finite and depth lemma.
Since injective dimension does not satisfy those properties, the proof of the main theorem of this paper requires a different argument from that.

\begin{thm}\label{main thm of fid}
Let $R$ be a ring, $I$ an ideal of $R$, and $M$ a finitely generated $R$-module.
Suppose that $\gor(R/\p)$ contains a nonempty open subset of $\spec (R/\p)$ for any $\p\in\supp_R(M)\cap\V(I)$.
Then there is an integer $k\ge 0$ such that for all $n\ge k$ and all prime ideals $\p$ of $R$,
$$
\id_{R_\p} (M/I^n M)_\p=\id_{R_\p} (M/I^k M)_\p.
$$ 
\end{thm}

\begin{proof}
The Rees ring $A=\bigoplus_{i\ge 0}I^i$ is a homogeneous noetherian graded ring and the associated graded module $N=\bigoplus_{i\ge 0}I^i M/I^{i+1} M$ is a finitely generated graded $A$-module.
Note that for any integer $t$, $\supp_R(\bigoplus_{i\ge t}I^i M/I^{i+1} M)$ is contained in $\supp_R(M)\cap\V(I)$.
There exists an integer $l\ge 0$ such that $X :=\fid_R(I^n M/I^{n+1} M)=\fid_R(I^l M/I^{l+1} M)$ for all $n\ge l$ by Corollary \ref{cor fid locus2} and Proposition \ref{homog comp fid locus}.
Now $X$ is an open subset of $\spec (R)$.
It follows from \cite[Corollary 8]{Ko} that there is an integer $k\ge l$ such that $\id_{R_\q} (M/I^n M)_\q=\id_{R_\q} (M/I^k M)_\q$ for all $n\ge k$ and all minimal elements $\q$ of $\spec(R)\setminus X$.
Put $N=\bigoplus_{i\ge k}M/I^i M$.
Let $\q$ be a minimal element of $\spec(R)\setminus\fid_R(N)$.
We claim that  for all $n\ge k$, $\id_{R_\q} (M/I^n M)_\q=\id_{R_\q} (M/I^k M)_\q$.
For each $n\ge k$, consider the exact sequence 
\begin{equation}\label{main thm of fid1}
0 \to (I^n M/I^{n+1} M)_\q \to (M/I^{n+1} M)_\q \to (M/I^n M)_\q \to 0.
\end{equation}
By (\ref{main thm of fid1}), we see that $X$ contains $\fid_R(N)$.
If $\q\in X$, then it is seen that $\id_{R_\q} (I^n M/I^{n+1} M)_\q<\infty$ for each $n\ge k$, and thus the claim follows from (\ref{main thm of fid1}).
Otherwise, since $\q$ is minimal in $\spec(R)\setminus\fid_R(N)$, so is it in $\spec(R)\setminus X$, and hence the claim holds.
Applying Lemma \ref{Unified lemma}, we get $\fid_R(N)=\fid_R(M/I^n M)$ for all $n\ge k$, which means that the assertion holds; see Lemma \ref{id sum}(2).
\end{proof}

We recall a few definitions of notions used in our next result.

\begin{dfn}\label{def rings}
A ring $R$ is said to be \textit{quasi-excellent} if the following two conditions are satisfied.
\begin{enumerate}[\rm(1)]
\item  For all finitely generated $R$-algebras $S$, $\reg(S) = \{\p\in\spec (S)\mid$ the local ring $S_\p$ is regular$ \}$ is open.
\item All the formal fibers of $R_\p$ are regular for all prime ideals $\p$ of $R$.
\end{enumerate}
A ring $R$ is said to be \textit{excellent} if it is quasi-excellent and universally catenary.
A ring in which ``regular'' is replaced with ``Gorenstein'' in both conditions (1) and (2) in the definition of an excellent ring is called an \textit{acceptable ring} \cite{S}.
\end{dfn}

It is well known that a complete ring is excellent and that an excellent ring and a homomorphic image of a Gorenstein ring are both acceptable.
Cases (1) and (2) of Theorem \ref{main1 id stable} are obtained as a corollary of Theorem \ref{main thm of fid}.
If an $R$-module $M$ has finite injective dimension, then $\fid_R (M)=\spec (R)$, but the converse does not hold in general.
Indeed, an infinite dimensional Gorenstein ring $R$ has infinite injective dimension and $\fid_R (R)=\gor(R)=\spec (R)$; see \cite[Proposition 9]{Fu} for an example of such a ring.

\begin{cor}\label{main cor of fid}
Let $R$ be a ring and $I$ an ideal of $R$.
Let $M$ be a finitely generated $R$-module.
Put $\bar{R}=R/(I+\operatorname{Ann}_R (M))$.
Then there is an integer $k>0$ such that
$$
\id_{R_\p} (M/I^n M)_\p=\id_{R_\p} (M/I^k M)_\p
$$ 
for all integers $n\ge k$ and all prime ideals $\p$ of $R$ in each of the following cases.\\
{\rm(1)} $\bar{R}$ is acceptable.
{\rm(2)} $\bar{R}$ is quasi-excellent.
{\rm(3)} $\bar{R}$ is excellent.
{\rm(4)} $\bar{R}$ is a homomorphic image of a Gorenstein ring. 
{\rm(5)} $\fid_{R/J} (M)=\spec (R/J)$ for some ideal $J$ of $R$ which is contained in $\operatorname{Ann}_R (M)$.
{\rm(6)} $\fid_{R/J} (M/I^n M)=\spec (R/J)$ for some $n>0$ and some ideal $J$ of $R$ contained in $\operatorname{Ann}_R (M/I^n M)$.
\end{cor}

\begin{proof}
In any of the former four cases, the assertion follows immediately from Theorem \ref{main thm of fid}.
In the latter two cases, for any $\q\in\supp_{R}(M)\cap\V(I)$, 
$\q$ contains $J$ and $(R/J)/(\q/J)$ is isomorphic to $R/\q$.
Applying Theorem \ref{fid locus extension} to $A=A_0=R/J$ and $\p=\q/J$, we see that $\gor(R/\q)$ contains a nonempty open subset of $\spec (R/\q)$ by assumption.
The assertion follows from Theorem \ref{main thm of fid}.
\end{proof}

In the main results of \cite{Ki}, the (semi-)local case was immediately deduced from other cases.
Indeed, since the depth of fiber rings takes a finite value, the argument could be reduced to the complete case using \cite[Proposition 1.2.16]{BH}.
A similar formula \cite[Corollary 1]{FT} exists for injective dimension, but the same argument does not work as the injective dimension of fiber rings is not necessarily finite.
So, we aim to apply \cite[Theorem]{FT} instead of \cite[Corollary 1]{FT}.
For this purpose, we consider Bass numbers in the next section.

\section{Bass numbers and betti numbers}

Throughout this section, let $R$ be a ring, $I$ an ideal of $R$, and $M$ a finitely generated $R$-module.
In this section, we study the asymptotic behavior of Bass numbers and Betti numbers of $M/I^n M$.
For prime ideals $\p, \q$ of $R$ such that $\p\subseteq\q$ and an integer  $n\ge 0$, there are inequalities $\beta_n (\p,M)\le \beta_{n} (\q,M)$ and $\mu^n (\p,M)\le \mu^{n+\height(\q/\p)} (\q,M)$; see \cite[Theorem 5.1]{F}.
The following lemma asserts that equality holds if some assumptions are imposed.

\begin{lem}\label{sp seq bass betti}
Suppose that $(R, \m, k)$ is local.
Let $n\ge 0$ be an integer, and $\p\in\spec (R)$.
\begin{enumerate}[\rm(1)]
\item  Put $d=\dim (R/\p)$.
Suppose that $\Ext_{R}^i (R/\p, M)=0$ is a free $R/\p$-module for all $0\le i\le n+d$ and that $R/\p$ is Gorenstein.
Then $\mu^n (\p,M)=\mu^{n+d} (\m,M)$.
\item Suppose that $\Tor_i^{R} (R/\p, M)=0$ is a free $R/\p$-module for all $0\le i\le n$.
Then $\beta_n (\p,M)=\beta_{n} (\m,M)$.
\end{enumerate}
\end{lem}

\begin{proof}
(1): There exists a spectral sequence $E_{2}^{p,q}=\Ext_{R/\p}^p (k, \Ext_{R}^q (R/\p, M))\Rightarrow H^{p+q}=\Ext_{R}^{p+q} (k, M)$.
Put $t=\mu^n (\p,M)$.
We have $E_{2}^{p,q}=0$ for any $p\ne d$ and any $0\le q\le n+d$ by assumption.
Hence we get $H^{n+d}\simeq E_{2}^{d,n}\simeq k^{\oplus t}$ as $R/\p$ is Gorenstein.

(2): There exists a spectral sequence $E_{p,q}^{2}=\Tor_p^{R/\p} (k, \Tor_q^{R} (R/\p, M))\Rightarrow H_{p+q}=\Tor_{p+q}^{R} (k, M)$.
The assertion can be shown in a similar way as in the proof of (1).
\end{proof}

The result below is useful to apply Lemma \ref{sp seq bass betti}.

\begin{lem}\label{free graded ext and tor}
Let $i\ge 0$ be an integer, and $\p$ a prime ideal of $R$.
Then there is $f\in R\setminus\p$ such that for any integer $n$, $\Ext_{R}^i (R/\p, M/I^n M)_f$ and $\Tor_i^{R} (R/\p, M/I^n M)_f$ are projective as $R_f/\p R_f$-modules.
\end{lem}

\begin{proof}
Let $A=\bigoplus_{n\ge 0}I^n$ be the Rees ring.
There is the natural exact sequence 
$$
0 \to 
\bigoplus_{n\ge 0} I^{n} M/I^{n+1} M \to 
\bigoplus_{n\ge 0}M/I^{n+1} M \to 
\bigoplus_{n\ge 0}M/I^n M 
\to 0
$$
of graded $A$-modules.
By Remark \ref{graded ext and f.g.}, we have an exact sequence
\begin{align*}
\bigoplus_{n\ge 0}\Ext_{R}^i (R/\p, I^{n} M/I^{n+1} & M) \to
\bigoplus_{n\ge 0}\Ext_{R}^i (R/\p, M/I^{n+1} M) \\
& \xrightarrow{g} \bigoplus_{n\ge 0}\Ext_{R}^i (R/\p, M/I^{n} M) \to
\bigoplus_{n\ge 0}\Ext_{R}^{i+1} (R/\p, I^{n} M/I^{n+1} M) 
\end{align*}
of graded $A/\p A$-modules.
Note that since $\bigoplus_{n\ge 0}\Ext_{R}^j (R/\p, I^{n} M/I^{n+1} M)$ are finitely generated $A/\p A$-modules for all $j\ge 0$, so are $\ker{g}$ and $\cok{g}$.
As $A/\p A$ is finitely generated $R/\p$-algebra, it follow from \cite[Lemma 8.1]{HR} that there is $f\in R\setminus\p$ such that $(\ker{g})_f$ and $(\cok{h})_f$ are free as $R_f/\p R_f$-modules.
There is an exact sequence 
$$
0 \to (\ker{g})_f \to \bigoplus_{n\ge 0}\Ext_{R}^i (R/\p, M/I^{n+1} M)_f \to \bigoplus_{n\ge 0}\Ext_{R}^i (R/\p, M/I^{n} M)_f \to (\cok{g})_f \to 0
$$
of graded $A_f/\p A_f$-modules.
Considering each homogeneous part, by induction on $n$, we conclude for all $n$ that $\Ext_{R}^i (R/\p, M/I^{n} M)_f$ are projective $R_f/\p R_f$-modules.
A dual proof works for the Tor modules.
\end{proof}

Combining the above two lemmas, we obtain the following corollary.
It plays an essential role in the proof of Proposition \ref{finite poly bass gor}, which is one of the main results of this section.

\begin{cor}\label{open bass and betti}
Let $i\ge 0$ be an integer, and $\p$ a prime ideal of $R$.
\begin{enumerate}[\rm(1)]
\item  Suppose that $\gor(R/\p)$ contains a nonempty open subset of $\spec (R/\p)$ and $\dim (R/\p)<\infty$.
Then there exists an open subset $U$ of $\spec(R)$ such that $\p\in U$ and 
$\mu^i (\p,M/I^n M)=\mu^{i+\height(\q/\p)} (\q,M/I^n M)$ for all $\q\in U\cap\V(\p)$ and all $n>0$.
\item There exists an open subset $U$ of $\spec(R)$ such that $\p\in U$ and  $\beta_i (\p,M/I^n M)=\beta_i (\q,M/I^n M)$ for all $\q\in U\cap\V(\p)$ and all $n>0$.
\end{enumerate}
\end{cor}

\begin{proof}
(1): We can choose $f\in R\setminus\p$ such that $R_f/\p R_f$ is Gorenstein, and $\Ext_{R}^j (R/\p, M/I^n M)_f$ are projective as $R_f/\p R_f$-modules for all $0\le j\le i+\dim (R/\p)$ and all $n>0$ by Lemma \ref{free graded ext and tor}.
Then $\p$ belongs to the open subset $U=\D(f)$ of $\spec(R)$.
The assertion follows from Lemma \ref{sp seq bass betti}.

(2): It can be shown in a similar way as in the proof of (1).
\end{proof}

In Corollary \ref{open bass and betti}(1), $\height(\q/\p)$ appears as a superscript, however, we would like to write the Bass numbers of $M/I^n M$ with respect to $\q$ without using the terms of other prime ideal $\p$.
Proposition \ref{finite poly bass gor} does makes it happen.

\begin{prop}\label{finite poly bass gor}
Let $i$ be an integer.
Suppose that $\bar{R}=R/(I+\operatorname{Ann}_R (M))$ is a finite-dimensional catenary ring, and $\gor(\bar{R}/\p\bar{R})$ contains a nonempty open subset of $\spec (\bar{R}/\p\bar{R})$ for any $\p\in\supp_R(\bar{R})$.
Then there exist polynomials $\varphi_1, \ldots, \varphi_l \in\Q[x]$ and $k>0$ such that the following condition is satisfied:
for any $\p\in\supp_R(\bar{R})$, there is $1\le j\le l$ such that $\mu^{i+\height(\p\bar{R})} (\p,M/I^n M)=\varphi_j (n)$ for all $n\ge k$.
\end{prop}

\begin{proof}
Put $X_0=\spec(R)\setminus\supp_R(\bar{R})$.
Since any ascending chain of open subsets of $\spec(R)$ stabilizes, it suffices to show the following claim.
Indeed, if the claim holds, then for some $l\ge 0$, there are polynomials $\varphi_1, \ldots, \varphi_l \in\Q[x]$, positive integers $k_1, \ldots, k_l$, and an ascending chain of open subsets $X_0\subsetneq X_1, \subsetneq \ldots, \subsetneq X_l=\spec(R)$ of $\spec(R)$ such that for any $1\le j\le l$, any $n\ge k_j$, and any $\q\in X_j\setminus X_{j-1}$, the equality $\mu^{i+\height(\q\bar{R})} (\q,M/I^n M)=\varphi_j (n)$ holds.
Put $k={\rm max}\{k_1, \ldots, k_l \}$.
Then for any $\p\in\supp_R(\bar{R})$, there is $1\le j\le l$ such that $\p\in X_j\setminus X_{j-1}$, and hence for any $n\ge k$, we have $\mu^{i+\height(\p\bar{R})} (\p,M/I^n M)=\varphi_j (n)$.
\begin{spacing}{1.3}
\end{spacing}
\noindent \textbf{Claim.} 
Let $X$ be an open subset of $\spec(R)$.
If $X_0\subseteq X\subsetneq\spec(R)$, then there is an open subset $Y$ of $\spec(R)$, a polynomial $\varphi\in\Q[x]$, and $k>0$ such that $X\subsetneq Y$, and $\mu^{i+\height(\q\bar{R})} (\q,M/I^n M)=\varphi (n)$ for all $n\ge k$ and all $\q\in Y\setminus X$.
\begin{spacing}{1.3}
\end{spacing}
\noindent \textit{Proof of Claim.}
We write $X=\D(J)$.
Since $J\ne R$, there is a minimal prime ideal $\p$ of $J$.
Note that $\p$ is not in $X$ and thus belongs to $\supp_R(\bar{R})$.
It follows from \cite[Corollary 7]{Ko} that there is a polynomial $\varphi\in\Q[x]$ and $k>0$ such that $\mu^{i+\height(\p\bar{R})} (\p,M/I^n M)=\varphi (n)$ for all $n\ge k$.
We can take $f\in R\setminus\p$ such that $\sqrt{J R_f}=\p R_f$; see \cite[Lemma 2.7(4)]{K}.
Also, we choose $g\in R\setminus\p$ such that it belongs to any minimal prime ideal of $I+\operatorname{Ann}_R (M)$ which is not contained in $\p$.
For all $\q\in \D(g)\cap\V(\p)$, we have $\height(\p\bar{R})+\height(\q/\p)=\height(\q\bar{R})$ as $\bar{R}$ is catenary.
Corollary \ref{open bass and betti}(1) implies that there exists an open subset $U$ of $\spec(R)$ such that $\p\in U$ and 
$\mu^{i+\height(\p\bar{R})} (\p,M/I^n M)=\mu^{i+\height(\p\bar{R})+\height(\q/\p)} (\q,M/I^n M)$ for all $\q\in U\cap\V(\p)$ and all $n>0$.
Then $Y=X\cup (\D(f)\cap\D(g)\cap U)$ is open.
Since $\p$ is in $Y\setminus X$, we get $X\subsetneq Y$.
Let $\q\in Y\setminus X$.
It is seen that $Y\setminus X=\V(J)\cap(\D(f)\cap\D(g)\cap U)=\V(\p)\cap\D(g)\cap U$.
Therefore, the equality $\mu^{i+\height(\q\bar{R})} (\q,M/I^n M)=\varphi (n)$ holds for any $n\ge k$.
\end{proof}

The result below can be obtained from Proposition \ref{finite poly bass gor}.

\begin{cor}\label{finite poly bass acceptable}
Let $s\ge 0$ be an integer.
If $\bar{R}=R/(I+\operatorname{Ann}_R (M))$ is acceptable and of finite dimension, then there exist polynomials $\varphi_1, \ldots, \varphi_l \in\Q[x]$ and $k>0$ such that the following condition is satisfied:
for any prime ideal $\p$ of $R$, there exists $1\le j\le l$ such that $\mu^{s} (\p,M/I^n M)=\varphi_j (n)$ for all $n\ge k$.
\end{cor}

\begin{proof}
It follows from Proposition \ref{finite poly bass gor} that there are polynomials $\varphi_1, \ldots, \varphi_l \in\Q[x]$ and $k>0$ such that for any $\p\in\supp_R(\bar{R})$ and any integer $-\dim(\bar{R})\le i\le s$, there exists $1\le j\le l$ such that for all $n\ge k$, $\mu^{i+\height(\p\bar{R})} (\p,M/I^n M)=\varphi_j (n)$.
Let $\p$ be a prime ideal of $R$.
If $\p\in\supp_R(\bar{R})$, then $-\dim(\bar{R})\le s-\height(\p\bar{R})\le s$, and thus there is $1\le j\le l$ such that $\mu^{s} (\p,M/I^n M)=\varphi_j (n)$ for all $n\ge k$.
Otherwise, we have $\mu^{s} (\p,M/I^n M)=0$ for any $n>0$.
\end{proof}

If $R$ is a complete local ring, then the assumption in Corollary \ref{finite poly bass acceptable} is satisfied.
Below is the formula for Bass numbers mentioned at the end of Section 2.
We use it to investigate the Bass number of a module over a local ring that is not necessarily complete.

\begin{lem}\cite[Theorem]{FT}\label{FT bass}
Let $S$ be a ring, $\varphi: R\to S$ a flat ring homomorphism, and $\q$ a prime ideal of $S$.
Put $\p=\q\cap R$.
Then for all integers $n$, there is an equality 
$$
\mu_{S}^{n} (\q,M\otimes_{R} S) = \sum_{p+q=n} \mu_{R}^{p} (\p,M)
\mu_{S/\p S}^{q} (\q/\p S,S/\p S).
$$
\end{lem}

There is an upper bound on the Bass numbers of formal fibers for all prime ideals.
When $R$ is local, we use $\hat{R}$ and $\hat{M}$ to denote the completion of $R$ and $M$, respectively.

\begin{cor}\label{finite poly bass gor cor}
Let $i\ge 0$ be an integer.
Suppose that $R$ is local.
Then there is an integer $N>0$ such that 
$\mu_{\hat{R}/\p \hat{R}}^{i+\depth \hat{R}_\q/\p\hat{R}_\q} (\q/\p\hat{R},\hat{R}/\p\hat{R})\le N$ for all prime ideals $\q$ of $\hat{R}$ and $\p=\q\cap R$.
\end{cor}

\begin{proof}
Note that $r_{\hat{R}}(0, \hat{R}, \q)=\depth \hat{R}_\q$ for any prime ideal $\q$ of $\hat{R}$.
We can apply Proposition \ref{finite poly bass gor} to see that $\{\mu_{\hat{R}}^{i+\depth \hat{R}_\q} (\q,\hat{R}) \mid \q\in\spec(\hat{R}) \}$ is a finite set.
By Lemma \ref{FT bass}, we have an inequality
$$
\mu_{\hat{R}}^{i+\depth \hat{R}_\q} (\q,\hat{R}) \ge \mu_{R}^{\depth R_\p} (\p,R)\mu_{\hat{R}/\p \hat{R}}^{i+\depth \hat{R}_\q/\p\hat{R}_\q} (\q/\p\hat{R},\hat{R}/\p\hat{R})
\ge \mu_{\hat{R}/\p \hat{R}}^{i+\depth \hat{R}_\q/\p\hat{R}_\q} (\q/\p\hat{R},\hat{R}/\p\hat{R})
$$
for any prime ideal $\q$ of $\hat{R}$ and $\p=\q\cap R$.
\end{proof}

This corollary can be shown by the same proof using \cite[Theorem 5.1]{F} instead of Proposition \ref{finite poly bass gor}.
We are now ready to prove Theorem \ref{main2 finite poly bass and betti}(1) in case (ii).

\begin{thm}\label{finite poly bass local}
Let $s\ge 0$ be an integer.
Suppose that $R/(I+\operatorname{Ann}_R (M))$ is semi-local.
Then there exist polynomials $\varphi_1, \ldots, \varphi_l \in\Q[x]$ and $k>0$ such that the following condition is satisfied:
for any prime ideal $\p$ of $R$, there exists $1\le j\le l$ such that $\mu^{s} (\p,M/I^n M)=\varphi_j (n)$ for all $n\ge k$.
\end{thm}

\begin{proof}
Since $\mu^{s} (\p,M/I^n M)=0$ for all $\p\in\D_R(I+\operatorname{Ann}_R (M))$ and all $n>0$, and $R/(I+\operatorname{Ann}_R (M))$ is semi-local, in order to prove this theorem, we may replace $R$ by $R_{\m}$ for each maximal ideal of $R$ containing $I+\operatorname{Ann}_R (M)$.
We prove the theorem by induction on $s$.
It follows from the induction hypothesis and Corollaries \ref{finite poly bass acceptable} and \ref{finite poly bass gor cor} that there exist polynomials $\varphi_1, \ldots, \varphi_l, \phi_1, \ldots, \phi_m \in\Q[x]$ and integers $k\ge k_{\hat{R}}(I\hat{R}, \hat{M})$ and $N>0$ such that the following conditions are satisfied.
\begin{enumerate}[\rm(1)]
\item For any $\p\in\spec(R)$ and any integer $0\le p\le s-1$, there exists $1\le j\le l$ such that for all $n\ge k$, $\mu_R^p (\p,M/I^n M)=\varphi_j (n)$.
\item For any $\q\in\spec(\hat{R})$, there exists $1\le i\le m$ such that 
$\mu_{\hat{R}}^{s} (\q,\hat{M}/I^n \hat{M})=\phi_i (n)$ for all $n\ge k$.
\item $\mu_{\hat{R}/\p \hat{R}}^{q} (\q/\p\hat{R},\hat{R}/\p\hat{R})\le N$ for all prime ideals $\q$ of $\hat{R}$, $\p=\q\cap R$, and all integers $0\le q\le s$.
\end{enumerate}
Then the subset
$$
X:=\left\{ \frac{1}{d_0} \left(\phi_i-\sum_{q=1}^{s} d_{q} \varphi_{j(q)} \right) \middle\vert \ 1\le i\le m,\ d_0\ne 0,\ 0\le d_q\le N,\ 1\le j(q)\le l {\rm \ for \ any \ } 1\le q\le s 
\right\}
$$
of $\Q[x]$ is finite set.
Fix a prime ideal $\p$ of $R$.
We claim that there exists $\Phi\in X$ such that for all $n\ge k$, $\mu^{s} (\p,M/I^n M)=\Phi (n)$.
Since $\hat{R}$ is faithfully flat over $R$, there is a prime ideal $\q$ of $\hat{R}$ such that $\p=\q\cap R$ and $\q$ is a minimal prime ideal of $\p\hat{R}$.
Note that $\depth(\hat{R}/\p\hat{R})_\q=0$.
Lemma \ref{FT bass} yields that the equality
$$
\mu^{s} (\p,M/I^n M)=\frac{1}{\mu_{\hat{R}/\p \hat{R}}^{0} (\q/\p\hat{R},\hat{R}/\p\hat{R})}
\left(\mu_{\hat{R}}^{s} (\q,\hat{M}/I^n \hat{M})-
\sum_{q=1}^{s} 
\mu_{\hat{R}/\p \hat{R}}^{q} (\q/\p\hat{R},\hat{R}/\p\hat{R})
\mu^{s-q} (\p,M/I^n M)
\right)
$$
holds for any $n>0$.
The claim follows from the conditions (1), (2) and (3).
\end{proof}

The following result is a corollary of the above theorem.

\begin{cor}\label{main cor of fid local}
Suppose that $R/(I+\operatorname{Ann}_R (M))$ is semi-local.
Then there is an integer $k>0$ such that
$$
\id_{R_\p} (M/I^n M)_\p=\id_{R_\p} (M/I^k M)_\p
$$ 
for all integers $n\ge k$ and all prime ideals $\p$ of $R$.
\end{cor}

\begin{proof}
Applying Theorem \ref{finite poly bass local} to $s={\rm max}
\{\height\p \mid \p\in\V(I+\operatorname{Ann}_R (M)) \}+1$, it is seen that there exist polynomials $\varphi_0, \varphi_1, \ldots, \varphi_l \in\Q[x]$ and $k>0$ such that the following conditions are satisfied:
\begin{enumerate}[\rm(1)]
\item  $\varphi_0$ is the zero polynomial, and $\varphi_j (n)>0$ for any $1\le j\le l$ and any $n\ge k$.
\item For any prime ideal $\p$ of $R$, there is $0\le j(\p)\le l$ such that for all $n\ge k$, $\mu_R^s (\p,M/I^n M)=\varphi_{j(\p)} (n)$.
\end{enumerate}
Let $\p$ be a prime ideal of $R$.
If $\p\in\D_R(I+\operatorname{Ann}_R (M))$, then $(M/I^n M)_\p$ are zero for all $n>0$.
Otherwise, we have $\depth R_\p\le \height\p<s$.
It follows from \cite[Theorem 1.1]{FFGR} and \cite[Theorem 2]{Rob} that for any $n\ge k,$ $\id_{R_\p} (M/I^n M)_\p<\infty$ if and only if $j(\p)=0$.
The proof is now completed; see Lemma \ref{id sum}(2).
\end{proof}

Two simple examples of the asymptotic behavior of injective dimensions are presented.

\begin{ex}
Let $R=K \llbracket x, y, z, w \rrbracket/(xy-zw)$ be a quotient of a formal power series ring over a field $K$.
Take the ideal $I=xR$ of $R$ and the finitely generated $R$-module $M=R/wR$.
The ring $R$ is a local hypersurface of dimension 3 that has an isolated singularity.
Let $\p\in\supp_R(M)\cap\V_R(I)$.
If $\p$ is not a maximal ideal, then the $R_\p$ is regular.
We get $\id_{R_\p} (M/I^n M)_\p=\depth R_\p$ for all $n>0$.
On the other hand, for any $n>0$, the minimal free resolution of the $R$-module $M/I^n M\simeq R/(w, x^n)R$ is
$$
\cdots \xrightarrow{
\begin{pmatrix}
  y & z \\
  w & x
\end{pmatrix}
} R^2 \xrightarrow{
\begin{pmatrix}
  x & -z \\
  -w & y
\end{pmatrix}
} R^2 \xrightarrow{
\begin{pmatrix}
  y & z \\
  w & x
\end{pmatrix}
} R^2 \xrightarrow{
\begin{pmatrix}
  x^n & zx^{n-1} \\
  -w & y
\end{pmatrix}
} R^2 \xrightarrow{
\begin{pmatrix}
  w & x^n
\end{pmatrix}
} R \to 0.
$$
For all $n>0$, we have $\pd_{R} (M/I^n M)=\infty$, which means that $\id_{R} (M/I^n M)=\infty$ since $R$ is Gorenstein.
This says that the integer $k=1$ satisfies the assertion of Corollary \ref{main cor of fid local}.
\end{ex}

\begin{ex}
Let $R=K [x, y, z]/(x^m y, x^m z)$ be a quotient of a polynomial ring over a field $K$, where $m>0$.
Take the ideal $I=xR$ of $R$ and the finitely generated $R$-module $M=R$.
Let $\p\in\V_R(I)$.
The ring $R_\p$ is not Cohen--Macaulay if $\p=(x,y,z)R$.
Hence, we obtain $\id_{R_\p} (M/I^n M)_\p=\infty$ for all $n>0$.
Suppose $\p\ne (x,y,z)R$.
Put $S=K [x, y, z]$ and $\p=\q/(x^m y, x^m z)$ for some prime ideal $\q$ of $S$.
Since $(x^m y, x^m z)S_\q=x^m S_\q$, we see that
$$
R_\p\simeq S_\q/x^m S_\q, \quad
(M/I^{n}M)_\p\simeq R_\p/x^n R_\p\simeq
\begin{cases}
        {S_\q/x^n S_\q \quad (n<m)}\\
        {S_\q/x^m S_\q \quad (n\ge m)}.
\end{cases}
$$
As $R_\p$ is Gorenstein, we have equalities
$$
\pd_{R_\p} (M/I^n M)_\p=
\begin{cases}
        {\infty \quad (n<m)}\\
        {0 \quad (n\ge m)}
\end{cases}
{\rm and} \quad
\id_{R_\p} (M/I^n M)_\p=
\begin{cases}
        {\infty \quad (n<m)}\\
        {\depth R_\p \quad (n\ge m)}.
\end{cases}
$$
This says that the integer $k=m$ satisfies the assertion of Corollary \ref{main cor of fid}.
\end{ex}

The following result is a Betti number version of Theorem \ref{finite poly bass local}.
The proof is a dual of the proof of Proposition \ref{finite poly bass gor}.
Note the difference in assumptions and subscripts/superscripts between (1) and (2) in Corollary \ref{open bass and betti}.
Also, \cite[Theorem 3.5]{Ki} is not necessary for the proof of Theorem \ref{finite poly betti}.

\begin{thm}\label{finite poly betti}
Let $s\ge 0$ be an integer.
Then there exist polynomials $\varphi_1, \ldots, \varphi_l \in\Q[x]$ and $k>0$ such that the following condition is satisfied:
for any prime ideal $\p$ of $R$, there exists $1\le j\le l$ such that $\beta_{s} (\p,M/I^n M)=\varphi_j (n)$ for all $n\ge k$.
\end{thm}

For any integer $i\ge 0$, we set $\PD_i^R (M) = \{\p\in\spec (R)\mid \pd_{R_\p}M_\p\le i \}$.
It is known that for any $i\ge 0$, $\PD_i^R (M)$ is an open subset of $\spec (R)$.
There is a projective dimension version of Corollary \ref{main cor of fid local}.

\begin{cor}\label{main thm of pd}
There is an integer $k>0$ such that for all integers $n\ge k$ and all prime ideals $\p$ of $R$,
$$
\pd_{R_\p} (M/I^n M)_\p=\pd_{R_\p} (M/I^k M)_\p.
$$ 
\end{cor}

\begin{proof}
For any $t>0$, we put $N_t=\bigoplus_{n\ge t} M/I^n M$.
Applying Theorem \ref{finite poly betti} to $s=i+1$ for each $i\ge 0$, an analogous argument to the proof of Corollary \ref{main cor of fid local} shows that there exists $k_i>0$ such that $\PD_i^R (M/I^n M)=\PD_i^R (M/I^{k_i} M)$ for all integers $n\ge k_i$.
The subset
$$
\PD_i^R (N_t)=\bigcap_{n\ge t}\PD_i^R (M/I^n M)=\bigcap_{n=t}^{{\rm max}\{t, k_i\}} \PD_i^R (M/I^n M)
$$
of $\spec(R)$ is open for any $t>0$ and $i\ge 0$.
Then $\PD_i^R (N_t)$ is contained in both $\PD_i^R (N_{t+1})$ and $\PD_{i+1}^R (N_t)$ for all $t>0$ and all $i\ge0$.
Since $R$ is noetherian, there are integers $k'>0$ and $m\ge 0$ such that the following conditions are satisfied; see \cite[Lemma 2.8]{Ki} for instance.
\begin{enumerate}[\rm(1)]
\item For any $t\ge k'$ and $i\ge 0$, the equality $\PD_i^R (N_t)=\PD_i^R (N_{k'})$ holds.
\item For any $i\ge m$, the equality $\PD_i^R (N_{k'})=\PD_m^R (N_{k'})$ holds.
\end{enumerate}
By \cite[Corollary 8]{Ko}, we can choose $k\ge k'$ such that for all $n\ge k$, all $0\le i\le m$, and all minimal elements $\q$ of $\spec(R)\setminus \PD_i^R (N_{k'})$, the equalities $\pd_{R_\q} (M/I^n M)_\q=\pd_{R_\q} (M/I^k M)_\q$ hold.
For all $i\ge 0$, applying Lemma \ref{Unified lemma}, we see that the equalities $\PD_i^R (N_k)=\PD_i^R (M/I^n M)$ hold for all $n\ge k$.
\end{proof}

\begin{ac}
The author would like to thank his supervisor Ryo Takahashi for valuable comments.
The auther is also indebted to Souvik Dey for pointing out that the previous version of this paper contains an already known result, which is removed in the current version.
\end{ac}

\end{document}